\documentclass[preprint,12pt]{elsarticle}

\usepackage{amssymb}
\usepackage{amsmath}
\usepackage{amsthm}
\usepackage{enumerate}
\usepackage{graphicx}
\usepackage{amsfonts}

\journal{...}

\newcommand*{\R}{\ensuremath{\mathbb{R}}}

\newtheorem{theorem}{Theorem}[section]

\newtheorem{prop}{Proposition}[section]
\newtheorem{remark}{Remark}[section]

\numberwithin{equation}{section}

\theoremstyle{definition}

\begin{document}

	\begin{frontmatter}

		\title{Remarks on a Bernstein-type operator of Aldaz, Kounchev and Render}

		\author[1]{Ana-Maria Acu}
		\author[2]{Heiner Gonska}
		\author[3]{Margareta Heilmann}

		\author[]{\hspace{11cm} {\bf Dedicated to professor Ioan Rasa on the occasion of his 70th birthday}}
		
		\address[1]{Lucian Blaga University of Sibiu, Department of Mathematics and Informatics, Str. Dr. I. Ratiu, No.5-7, RO-550012  Sibiu, Romania, e-mail: anamaria.acu@ulbsibiu.ro}
				\address[2]{University of Duisburg-Essen, Faculty of Mathematics, Bismarckstr. 90, D-47057 Duisburg, Germany, e-mail: heiner.gonska@uni-due.de }
		\address[3]{University of Wuppertal, School of Mathematics and Natural Sciences,
			Gau{\ss}stra{\ss}e 20,
			D-42119 Wuppertal, Germany, 	e-mail: heilmann@math.uni-wuppertal.de}

		\begin{abstract} 		
		The Bernstein-type operator of Aldaz, Kounchev and Render (2009) is discussed. New direct results in terms of the classical second order modulus as well as in a modification following Marsden and Schoenberg are given.
		\end{abstract}
	

	\end{frontmatter}
	

		
\section{Introduction} 

This brief note is dedicated to our dear friend and long term collaborator {\bf Ioan Rasa }on the occasion of his 70th birthday. Over the years he did a lot, most fruitful work on quite a number of Bernstein-type and related operators, among others. Here we deal with a then surprising, 12-year old object of this type offering a lot of challenges.

\vspace {5mm}

		Starting from the classical Bernstein operators $B_n$ defined for $f\in C[0,1]$ as
$$ B_n(f;x)=\displaystyle\sum_{k=0}^nf\left(\dfrac{k}{n}\right)p_{n,k}(x),\,\, \text{ where}$$
$$ p_{n,k}(x)={n\choose k}x^k(1-x)^{n-k}, $$
during recent years many, probably much too many, modifications have been considered. One of the present hypes follows a 2003 paper written by J.P. King \cite{1} in order to obtain linear positive operators which preserve two functions different from the classical test functions $e_0$ and $e_1$. Here we used the convention $e_j(x)=x^j,\,\, j=0,1,\dots$

King  modified the classical Bernstein operators as follows:
\begin{equation}\label{e1} f\to \left(B_nf\right)\circ r_n,\,\, f\in C[0,1], \end{equation}
where
$$r_n(x)=\left\{\begin{array}{ll}
x^2,& n=1,\\
\frac{1}{2(n-1)}\left(-1+\sqrt{1+4n(n-1)x^2}\right), & n=2,3,\dots\\
\end{array}\right. $$

These operators preserve the functions $e_i(x)=x^i$ for $i=0,2$. A slight extension was considered by C\'ardenas et al. in \cite{CGM2006} 
where  a sequence of operators $B_{n,\alpha}$ that preserve $e_0$ and $e_2+\alpha e_1$, $\alpha\in [0,+\infty)$ was introduced.
It is clear that the operator of King does not produce polynomials in the general case. Therefore, in their 2005 remarks on the article of King, Gonska and Pi\c tul (see \cite{GoPi2005}) asked the question if there exist linear and positive {\bf polynomial} operators that reproduce $e_2$.

Using a continuous strictly increasing function $\tau$ defined on $[0,1]$ with $\tau(0)=0$ and $\tau(1)=1$, $\tau'(x)>0, x\in  [0,1]$, C\'ardenas-Morales et al. (see  \cite{CGM2006}) introduced a modification of the Bernstein operator as follows:

\begin{equation}\label{e2} 
B_n^{\tau}(f;x):=\displaystyle\sum_{k=0}^n {n\choose k}\tau(x)^k(1-\tau(x))^{n-k}(f\circ \tau^{-1})\left(k/n\right), \,\, f\in C[0,1],\,\, x\in[0,1]. 
\end{equation}

Note that $\tau'(0) > 0$ is essential for their results.
A predecessor can be found in Cottin et al. \cite{CGGKZ1999}, but $\tau (x)= \sqrt{x}$  there.
\\

These operators preserve the functions $e_0$ and $\tau$. So, for $\tau(x)=x^j$, one has
$$ B_n^{\tau}(e_0;x)=e_0 \textrm{ and } B_n^{\tau}(e_j;x)=e_j.$$ In this case $B_n^{\tau}$ produces polynomials of degree $jn.$
If the function $\tau$ is not a polynomial, then the modified operator (\ref{e2}) is not a polynomial operator. 

\vspace{2mm}


Fix $j>1$, $j\in {\mathbb N}$. Let  $\Pi_n$ be the space of polynomials over $[0,1]$ of degree less than or equal to $n$. For every  $n\geq j$,  Aldaz, Kounchev and Render \cite{AKR2009} introduced a {\bf polynomial} Bernstein operator $B_{n,0,j}: C[0,1]\to \Pi_n$ that fixes $e_0$ and $e_j$, and converges in the strong operator topology to the identity as $n\to \infty$. The operator is a linear combination of the classical Bernstein basis $\{p_{n,k}\}_{k=0,\dots,n}$, thus produces polynomials of degree $n$ and is explicitly given by
\begin{equation}\label{e4} B_{n,0,j}(f;x)=\displaystyle\sum_{k=0}^n f\left(t_{n,k}^j\right)p_{n,k}(x),\end{equation}
where $$t_{n,k}^j=\left(\dfrac{k(k-1)\dots(k-j+1)}{n(n-1)\dots(n-j+1)}\right)^{1/j}.$$

\begin{remark}
	\begin{itemize}
	\item[(i)] For $j=1$ this is the classical Bernstein operator (reproducing $e_0$ and $e_1$).
	
		\item[(ii)] $B_{n,0,j}$ is linear, positive and of the form
	\begin{equation}
	\label{e3} \displaystyle\sum_{k=0}^nl_{n,k}(f) p_{n,k}(x),
	\end{equation}
where	$l_{n,k}$ are positive linear functionals with $ l_{n,k}(e_0)=1$.
Obviously all the $B_{n,0,j}$s map into $\Pi_n[0,1]$ (in contrast to the $B_n^{\tau}$s from above, case $\tau = e_j$).

\item[(iii)] For $j$ fixed and $n=j$ we have $(t_{n,k}^j)_{k=0,\dots,n}=(0,\dots,0,1)$ with $n$ zeros preceding the $1$.

For $n=j+1$ one obtains $(t_{n,k}^j)_{k=0,\dots,n}=\left(0,\dots,0,n^{\frac1 j},1\right)$, and so on. That is, only for  $n$ big enough all the nodes will be distinct.
\item[(IV)] The notation $B_{n,0,j}$ is motivated by the fact that the operator reproduces $e_0$ and $e_j$. As shown by Finta in  \cite{Fin2013}, see Theorem 2.1 there, no sequence $(L_n)$ of type (\ref{e3}) can reproduce $e_i$ and $e_j$, $i,j\in\{1,2,\dots\},\,\, i<j$. So reproduction of $e_0$ is a must!

Moreover, there exist infinitely many sequences of operators $L_n$ of type (\ref{e3}) which approximate each continuous function on $[0, 1]$,
and have the functions $e_0$ and $e_j$ as fixed points, where $j\in \{1, 2,\dots\}$ is given (see \cite[Th. 2.4]{Fin2013}).

\end{itemize}
	
\end{remark}

\section{On direct estimates for $B_{n,0,j}$}

We start with some brief history.
For the case $j=2$, $n\geq 2,$ an inequality involving $\omega_1$ (following Shisha $\&$ Mond) was given by Cardenas-Morales et al. in \cite[Proposition 3.1]{CGR2012} as follows
$$|B_{n,0,2}f(x)-f(x)|\leq \omega_1(f,\delta)\left(1+\dfrac{1}{\delta}\sqrt{2x(1-x)\dfrac{1-(1-x)^{n-1}}{n-1}}\right),$$
$ f\in C[0,1],\,x\in[0,1]\textrm{ and } \delta>0.  $


Moreover, in 2014 Finta showed (see \cite{Fin2014})
$$ |B_{n,0,j}(f;x)-f(x)|\leq 2\omega_1\left(f; \left(\dfrac{2j(3+4j)}{\sqrt{n}}\right)^{\frac{1}{2j}}\right),$$
$ f\in C[0,1],\, x\in [0,1] \textrm{ and } n\geq j\geq 2.$

In 2018 Aldaz and Render considered a generalization of the classical Bernstein operators on the polynomial spaces, which reproduce ${\bf 1}$ and a polynomial $f_1$, strictly increasing on $(0,1)$. Denote by $B_n^{f_1}$ this Bernstein type operator. For $f\in C[0,1]$, $x\in [0,1]$ (see \cite[Remark following Th. 6.5]{AlRe2018})
$$| B_n^{f_1}(f;x)-f(x)|\leq (c_S+1)\omega_1\left(f, n^{-\frac{1}{2}}\right),   $$
where $c_S=\dfrac{4306+837\sqrt{6}}{5832}$, Sikkema's constant. The drawback here is that the latter is shown to be valid only for $n$ sufficiently large.

\vspace{2mm}

We will improve the estimates known so far in three different ways.

\begin{prop} \label{p2.1} Let $f\in C[0,1]$, $x\in [0,1]$, $n \geq j \geq 1$. Then
	\begin{equation} \label{D} |f(x)-B_{n,0,j}(f;x)|\leq 1\cdot\omega_2\left(f;\dfrac{1}{\sqrt{n}}\right)+\omega_1\left(f;\dfrac{j-1}{n}\right).   \end{equation}
	\end{prop}
\begin{proof} We have 
		\begin{align*}|f(x)-B_{n,0,j}(f;x)|\\
		& \leq |f(x)-B_n(f;x)|+|B_n(f;x)-B_{n,0,j}(f;x)|\\
		&\leq 1\cdot\omega_2\left(f;\dfrac{1}{\sqrt{n}}\right)+\omega_1\left(f;\dfrac{j-1}{n}\right).    \end{align*}
		The first summand is taken from P\u alt\u anea  \cite{Pal2003,Pal2004}, while the second was given by Acu and Ra\c sa \cite{AcRa2016}. For $j=1$ this reduces to a best possible result for classical Bernstein operators.
	\end{proof}
	
	There is a second way to prove a similar inequality. To this end we  use P\u alt\u anea's general theorem for positive linear operators reproducing constant functions (see \cite[Cor.2.2.1, p.31]{Pal2004}).
\begin{align*} 
& |B_{n,0,j}(f;x)-f(x)|\leq \dfrac{1}{h} |B_{n,0,j}(e_1-x;x)|\cdot\omega_1(f;h)\\
&+\left(1+\dfrac{1}{2h^2}B_{n,0,j}\left((e_1-x)^2;x\right)\right)\cdot\omega_2(f;h),\,\, 0<h\leq\frac12.
\end{align*}

\begin{prop}\label{p2.2} Let $f\in C[0,1]$, $x\in [0,1]$, $1 \leq j \leq n$.  Then
	\begin{itemize}
		\item[a)] $|B_{n,0,2}(f;x)\!-\!f(x)|\!\leq\! \dfrac{1}{\sqrt{n-1}} d_n(x)\omega_1\left(f;\dfrac{1}{\sqrt{n\!-\!1}}\right)\!+\!\left(1\!+\!xd_n(x)\right)\omega_2\left(f;\dfrac{1}{\sqrt{n\!-\!1}}\right)$, $n\geq 5$, where $d_n(x):=(1-x)\left[1-(1-x)^{n-1}\right]$;
		\item[b)] $|B_{n,0,j}(f;x)-f(x)|\leq (j-1)\dfrac{1}{\sqrt{n}}\omega_1\left(f;\dfrac{1}{\sqrt{n}}\right)+\left(\dfrac{1}{8}+j\right)\omega_2\left(f;\dfrac{1}{\sqrt{n}}\right),\, n\geq 4. $
	\end{itemize}
\end{prop}
\begin{proof}
	a) In order to use P\u alt\u anea's result we  estimate $B_{n,0,2}(e_1-x;x)$ and\linebreak  $B_{n,0,2}((e_1-x)^2;x)$.
	
		According to \cite[(3.7)]{CGR2012} there holds
\begin{align*}|B_{n,0,2}(e_1-x;x)|&= |B_{n,0,2}(e_1;x)-x|\\
&\leq \dfrac{1-x}{n-1}\left(1-(1-x)^{n-1}\right) =\dfrac{d_n(x)}{n-1}.\end{align*}
	Furthermore, 
	\begin{align*} 0\leq B_{n,0,2}\left((e_1-x)^2;x\right)&=2x\left(x-B_{n,0,2}(e_1;x)\right)\\
	&\leq \dfrac{2xd_n(x)}{n-1}. \end{align*}
	Taking $h=\dfrac{1}{\sqrt{n-1}}$, we get the estimation a).
	
	In order to illustrate the nonsymmetry of the situation, in Fig.\ref{fig:1} are the graphs of the function $d_n(x)$ for $n=5$ and $n=10$, $x\in[0,1]$.
	
	\begin{figure}[htbp]
		\centering
		\includegraphics[height=70mm,keepaspectratio]{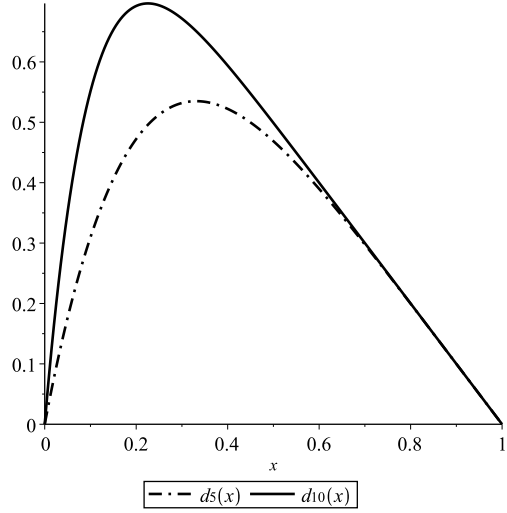}
		\caption{$d_n(x)$ for $n=5$ and $n=10$, $x\in [0,1]$} \label{fig:1}
	\end{figure}

	b) One has to determine global estimates for $|B_{n,0,j}(e_1-x;x)|$ and $B_{n,0,j}\left((e_1-x)^2;x\right)$, using again  $B_{n,0,j}(e_0;x)=1$.
	We have 
	\begin{align*}
	B_{n,0,j}(e_1-x;x)=\displaystyle\sum_{k=0}^n\left[\left(\dfrac{k}{n}\dots\dfrac{k-j+1}{n-j+1}\right)^{1/j}-\dfrac{k}{n}\right]p_{n,k}(x).
	\end{align*}
	The result from \cite[Section 4.7]{AcRa2016}, namely
\begin{equation*}\label{eq.rasa_ana}
-\dfrac{j-1}{n}\leq \left(\dfrac{k}{n}\dots\dfrac{k-j+1}{n-j+1}\right)^{1/j}-\dfrac{k}{n}\leq 0, \end{equation*}
	yields
	\begin{equation}\label{A}|B_{n,0,j}(e_1-x;x)|\leq \dfrac{j-1}{n}.\end{equation}
	 The second moment can be written as
	 \begin{align}\label{B}
	0&\leq  B_{n,0,j}\left((e_1-x)^2;x\right)\nonumber\\
	&=B_{n,0,j}(e_2;x)-x^2-2x\left[B_{n,0,j}(e_1;x)-x\right]\nonumber\\
	&\leq B_{n,0,j}(e_2;x)-x^2+2x\dfrac{j-1}{n}, 
	 \end{align}
	 where we used (\ref{A}).
	 
	\noindent We have
	 \begin{align*}
	& B_{n,0,j}(e_2;x)-x^2 \\
	&= B_{n,0,j}(e_2;x)-B_n(e_2;x)+B_n(e_2;x)-x^2\\
	 &=\displaystyle\sum_{k=0}^n\left[\left(\dfrac{k}{n}\cdots\dfrac{k-j+1}{n-j+1}\right)^{\frac{2}{j}}-\left(\dfrac{k}{n}\right)^2\right]p_{n,k}(x)+\dfrac{x(1-x)}{n}\\
	 &=\displaystyle\sum_{k=0}^n\left[\left(\dfrac{k}{n}\cdots\dfrac{k-j+1}{n-j+1}\right)^{\frac{1}{j}}-\dfrac{k}{n}\right]\left[\left(\dfrac{k}{n}\cdots\dfrac{k-j+1}{n-j+1}\right)^{\frac{1}{j}}+\dfrac{k}{n}\right]p_{n,k}(x)+\dfrac{x(1-x)}{n}\\
	 &\leq \dfrac{x(1-x)}{n}.
	 \end{align*}
	 Together with (\ref{B}) this yields
	 $$ 0\leq B_{n,0,j}\left((e_1-x)^2;x\right)\leq \dfrac{x(1-x)}{n}+2x\dfrac{j-1}{n}.  $$
	 	 So we arrive at
	 \begin{align}\label{C}
	 |B_{n,0,j}(f;x)-f(x)|&\leq \dfrac{j-1}{n}\dfrac{1}{h}\omega_1(f;h)+\left\{1+\dfrac{1}{2h^2}\left[\dfrac{x(1-x)}{n}+2x\dfrac{j-1}{n}\right]\right\}\omega_2(f;h)\nonumber\\
	 &\leq \dfrac{j-1}{n}\dfrac{1}{h}\omega_1(f;h)+\left\{1+\dfrac{1}{2nh^2}\left(\dfrac{1}{4}+2j\right)\right\}\omega_2(f;h).
	 \end{align}
	 Let $h=\dfrac{1}{\sqrt{n}}$, $n\geq 4$. This leads to our proposition.

	\end{proof}

\begin{remark}
	a) The right hand side of  inequality (\ref{C}) is not a pointwise one and so does not express interpolation at the endpoints. There is room for improvement.
	
	b) If one compares the estimate from Proposition \ref{p2.2} b) and the estimate from Proposition \ref{p2.1} it is clear that $\omega_2\left(f;\dfrac{1}{\sqrt{n}}\right)$ is always the worse (i.e., dominant) term. In case $f\in C^2[0,1]$ both inequalities give ${\mathcal O}\left(\dfrac{1}{n}\right)$. This cannot be reached by any of the previous inequalities in terms of $\omega_1$ only .
\end{remark}

	The inequalities (\ref{D}) and those from Proposition \ref{p2.2} do not imply that the function $e_j$ is reproduced by $B_{n,0,j}$ if $j\geq 2$. In order to achieve this we follow an idea of Marsden and Schoenberg  \cite [p.82] {Schoenberg} (see also \cite[Corollary 3.4]{ANTA}). Below we will use a variation of an analogue of  their $\omega_1^*(f;\delta)$, namely $ \omega_j^*(f;\delta)$ given by
	
	$$ \omega_j^*(f;\delta):=\displaystyle\inf_{l}\left\{\omega_1(f-l;\delta)\right\},\textrm{ where } l(x)=ae_j,\,\, a \in \mathbb{R}. $$
	Note that we won't loose anything in doing so!

\begin{prop} \label{p2.3}
For $f\in C[0,1]$, $x\in [0,1]$ and $n \geq j$ we have
\begin{align*}
	|f(x)-B_{n,0,j}(f;x)|\leq \displaystyle\inf_{l}\left\{\omega_2\left(f-l;\dfrac{1}{\sqrt{n}}\right)+\omega_1\left(f-l;\dfrac{j-1}{n}\right)\right\}.
	\end{align*}
\end{prop}
\begin{proof}
Since $B_{n,0,j}$ reproduces $e_j$ we have
	\begin{align*}
	|f(x)-B_{n,0,j}(f;x)|&=\left|(f-l)(x)-B_{n,0,j}(f-l;x)\right|\\
	&\leq 1\cdot \omega_2\left(f-l;\dfrac{1}{\sqrt{n}}\right)+\omega_1\left(f-l;\dfrac{j-1}{n}\right).
	\end{align*}
	Since $l=ae_j$ was arbitrary, we pass to the $inf$ and get the result.
	
	Moreover, if $f$ is of the form $f(x)=a_0+a_j x^j,\, a_0,a_j\in {\mathbb R}$, we arrive at
	$$|f(x)-B_{n,0,j}(f;x)|=0.  $$
\end{proof}
	Analogous statements follow from the inequalities in Proposition \ref{p2.2}.

\begin{remark}

	Except for the papers previously cited we did not find any reference mentioning or using $\omega_l^*$ or modifications thereof.  The similarity to K-functionals is obvious. 
	
	In  \cite{ANTA}	it is mentioned that
	
$$ \omega_1^*(f;\delta) \leq c \cdot \omega_2(f;\sqrt \delta).$$

\end{remark}

\section{Iterates of $ B_{n,0,j}$}
Using the approach of Agratini and Rus (see \cite{AgRu2003}) the following result for the sequence of iterates $(B_{n,0,j}^m)_{m\geq 1}$ is obtained.
\begin{theorem} For $n\geq j$ fixed and $m\geq 1$, the iterates sequence $(B_{n,0,j}^m)_{m\geq 1}$ satifies $$\displaystyle\lim_{m\to\infty}B_{n,0,j}^m(f;x)=f(0)+\left[f(1)-f(0)\right]e_j(x).$$
	\end{theorem}
\begin{proof}
Denote
$$ X_{\alpha,\beta}:=\left\{f\in C[0,1]\,\,\left|\,\, f(0)=\alpha,\,f(1)=\beta \right.\right\}, \,\, (\alpha,\beta)\in {\mathbb R}\times {\mathbb \R}. $$
The operators $B_{n,0,j}$ interpolate $0$ and $1$:
\begin{align*}
&B_{n,0,j}(f;0)=f(t_{n,0}^j)p_{n,0}(x)=f(0),\\
&B_{n,0,j}(f;1)=f(t_{n,n}^j)p_{n,0}(x)=f(1).
\end{align*}
Let $f,g\in X_{\alpha,\beta}$. Then, for $x\in [0,1]$,
\begin{align*}
|B_{n,0,j}(f;x)-B_{n,0,j}(g;x)|&=\left|\displaystyle\sum_{k=1}^{n-1}[f(t_{n,k})-g(t_{n,k})]p_{n,k}(x)\right|\\
&\leq \left(1-p_{n,0}(x)-p_{n,n}(x)\right)\cdot \|f-g\|_{\infty}\\
&\leq\left(1-\dfrac{1}{2^{n-1}}\right)\|f-g\|_{\infty}.\
\end{align*}
Therefore, $B_{n,0,j}|_{X_{\alpha,\beta}}:X_{\alpha,\beta}\to X_{\alpha,\beta} $ is a contraction for every $(\alpha, \beta)\in {\mathbb R}\times {\mathbb R}$ and $n\geq j$ fixed.

The function $p_{\alpha,\beta}^*:=f(0)+\left(f(1)-f(0)\right)e_j$ belongs to $X_{f(0),f(1)}$ and is a fixed point of $B_{n,0,j}$.

By contradiction we obtain the claim.

\end{proof}

The above result was also obtained by Gavrea and Ivan as an example for a  general theorem concerning the limit of the iterates of positive linear operators (see \cite{GaIv2011}). No corresponding quantitative version is known to us.

\section{Voronovskaya result for the Bernstein-type operators $B_{n,0,j}$}
 The asymptotic formula of the operator $B_{n,0,j}$ was stated as a conjecture by C\'ardenas-Morales et al. \cite{CGR2012} and proved by Birou in \cite{Bir2017}. 
 
 \begin{theorem}(see \cite{CGR2012},\cite{Bir2017} ) For all $f\in C[0,1]$, $x\in (0,1)$ and $j\geq 1$, whenever $f^{\prime\prime}(x)$ exists,
 	\begin{equation}
 	\label{e2v} \displaystyle\lim_{n\to\infty} n(B_{n,0,j}(f;x)-f(x))=\dfrac{x(1-x)}{2}f^{\prime\prime}(x)-\dfrac{(j-1)(1-x)}{2}f^{\prime}(x).
 	\end{equation}
 	\end{theorem}
 
 A quantitative {\bf pre-Voronovskaya theorem} for the Bernstein type operator  $B_{n,0,j}$ was proved by Finta \cite{Finta1} using the first order Ditzian-Totik modulus of smoothness
 defined by
 $$ \omega_{\varphi}^1(f;\delta)=\displaystyle\sup_{0<h\leq \delta}\sup_{x\pm \frac{1}{2}h\varphi(x)}\left|f\left(x+\frac{1}{2}h\varphi(x)\right)-f\left(x-\frac{1}{2}h\varphi(x)\right)\right|, $$
 where $\varphi(x)=\sqrt{x(1-x)}$, $x\in[0,1]$.
 We call this "pre-Voronovskaya" since it appears to point into a good direction but the limit is not explicitly given.
 \begin{theorem}
 \cite{Finta1} There exists $C>0$ depending only on $j$ such that
 $$ \left|n(B_{n,0,j}(f;x)-f(x))+f^{\prime}(x)nB_{n,0,j}(xe_0-e_1;x)-\frac{1}{2} f^{\prime\prime}(x)nB_{n,0,j}((e_1-xe_0)^2;x)\right|$$
 $$\leq C\omega_{\varphi}^1\left(f^{\prime\prime};\dfrac{1}{\sqrt{n}}\right), $$
 for all $x\in [0,1]$, $f\in C^2[0,1]$ and $n\geq j \geq 1$.
 \end{theorem}
 
 No quantitative version of a Voronovskaya-type theorem for the operators of Aldaz et al. is known to us.

\end{document}